\documentclass[11pt,fleqn]{amsart}

\usepackage{latexsym,amsmath,amssymb,amsfonts,MnSymbol,epsfig,verbatim}
\usepackage[initials]{amsrefs}
\usepackage[all]{xy}

\providecommand{\CC}{{\mathbb{C}}}
\providecommand{\RR}{{\mathbb{R}}}

\providecommand{\ZZ}{{\mathbb{Z}}}

\newtheorem{theorem}{Theorem}[subsection]
\newtheorem{lemma}[theorem]{Lemma}

\newtheorem{proposition}[theorem]{Proposition}

\theoremstyle{definition}

\theoremstyle{remark} 
\newtheorem{remark}[theorem]{Remark}

\numberwithin{equation}{section}

\begin{document}

\title{$K$-homology and Fredholm Operators II: Elliptic Operators}
\author{Paul F.\ Baum}
\address{The Pennsylvania State University, University Park, PA, 16802, USA}
\email{baum@math.psu.edu}
\author{Erik van Erp}
\address{Dartmouth College, 6188, Kemeny Hall, Hanover, New Hampshire, 03755, USA}\email{jhamvanerp@gmail.com}

\thanks{PFB was partially supported by NSF grant DMS-0701184}
\thanks{EvE was partially supported by NSF grant DMS-1100570}

\maketitle

\tableofcontents

\section{Introduction}

In this paper we prove commutativity of the triangle
\[ \xymatrix{   & K^0(T^*X) \ar[dr]^{\rm Op}\ar[dl]_c& \\
 K_0^{top}(X) \ar[rr]_\mu & & KK^0(C(X),\CC)              
 }
\]
for a closed smooth manifold $X$. This triangle was first introduced by Paul Baum and Ron Douglas in \cite{BD80} to formulate the Atiyah-Singer index theorem for elliptic operators in the framework of $K$-homology. 
An elliptic operator on a closed manifold $X$ determines an element in the {\em analytic} $K$-homology group $KK^0(C(X), \CC)$, and the solution of the index problem amounts to finding a {\em topological} description of this element.
As proposed in \cite{BD80}, a convenient way to formalize this problem is to ask for the construction of an explicit $K$-cycle in the {\em geometric} $K$-homology group $K_0^{top}(X)$ that (under the natural isomorphism $\mu$ between analytic and geometric $K$-homology) corresponds to the given elliptic operator.

Our renewed interest in the above triangle stems from its essential role in our solution of the index problem for the Heisenberg calculus on contact manifolds \cite{BvE3}.
In \cite{BD80} Baum and Douglas noted that commutativity of the triangle is implied by the Atiyah-Singer index theorem for {\em families} of elliptic operators.
In this paper we give a direct proof that does not rely on the  families index theorem.
Such a direct proof is needed in order to extend the results of \cite{BvE3} to the families case and the equivariant case.

Commutativity of the triangle is tantamount to the assertion that the index problem for elliptic operators reduces to the index problem for Dirac operators.
Hence, we provide a simple direct proof of reduction to the Dirac case.

In section \ref{triangle}, the groups and maps of the triangle are defined.
For the proof of commutativity, preliminaries are given in sections \ref{family} and \ref{rotation}.
The proof is given in section \ref{reduction}.

\section{The triangle}\label{triangle}

\subsection{$K$-theory (with compact supports)}

Throughout this paper $K$-theory, denoted $K^\bullet(Y)$ for a locally compact space $Y$, is Atiyah-Hirzebruch $K$-theory,
i.e. topological $K$-theory with compact supports.
Any element in $K^0(Y)$ is given by a triple $(\sigma, E, F)$ where $E$, $F$ are $\CC$ vector bundles on $Y$,
and $\sigma:E\to F$ is a vector bundle map which is an isomorphism outside a compact subset of  $Y$.

In particular, if $X$ is a closed smooth manifold, and $P:C^\infty(E^0)\to C^\infty(E^1)$ is an elliptic (pseudo)differential operator, then  the principal symbol $\sigma:\pi^*E^0\to \pi^*E^1$  determines an element in $K^0(T^*X)$.
Here $\pi:T^*X\to X$ is the projection.

\subsection{Geometric $K$-homology}

Geometric $K$-cycles were introduced in \citelist{\cite{BD80} \cite{BD82}}.
For a detailed description of the relation between geometric $K$-homology and analytic $K$-homology, see \cite{BHS10}.
In this section we briefly recall the definition and main features of $K$-cycle $K$-homology.
The theory is defined for the category of paracompact Hausdorff topological spaces. In particular, since any $CW$ complex is
paracompact Hausdorff, $K$-cycle $K$-homology is defined for any $CW$ complex. 

\vskip 6pt

Given a paracompact Hausdorff topological space $X$, a $K$-cycle for $X$ is a triple $(M,E,\varphi)$
consisting of a compact (without boundary) Spin$^c$ manifold $M$,
a smooth $\CC$ vector bundle $E$ on $M$,
and a continuous map $\varphi\,\colon M\to X$.
The collection of all such $K$-cycles, subject to a certain equivalence relation, forms an abelian group under disjoint union.
We denote this group by $K^{top}_\ast(X)$.

The equivalence relation that is imposed on the $K$-cycles is generated by three elementary steps: 
\begin{itemize}
\item bordism 
\item direct sum-dijoint union
\item vector bundle modification
\end{itemize}

The abelian group $K^{top}_*(X)$ consists of the equivalence classes of $K$-cycles.
Addition is given by disjoint union,
\[ (M_0,E_0,\varphi_0) + (M_0,E_0,\varphi_0) = (M_0\sqcup M_1,E_0\sqcup E_1,\varphi_0\sqcup \varphi_1)\]
and the additive inverse of a $K$-cycle is obtained by reversing the Spin$^c$ structure
\[ -(M,E,\varphi) = (-M,E,\varphi).\] 
The group $K^{top}_*(X)$ is $\ZZ/2$ graded by the parity of the dimension of the Spin$^c$ manifold $M$.
In other words, $K^{top}_0(X)$ and $K^{top}_1(X)$ consist of equivalence classes of $(M,E,\varphi)$ cycles for which every connected component of $M$ is even or odd dimensional respectively.

\subsection{The clutching map}

Let $X$ be a closed $C^\infty$ manifold, not required to be oriented or even dimensional.
Consider the Spin$^c$ manifold $\Sigma X = S(TX\times \RR)$, the unit sphere bundle of $TX\times \RR$.
$TX$ is an almost complex manifold (see Section \ref{rotation} below).
Therefore  $S(TX\times \RR)$ is a {\em stably} almost complex manifold,
and therefore is a Spin$^c$ manifold with Dirac operator $D_{\Sigma X}$.
Note that the dimension of the manifold $\Sigma X$ is two times the dimension of $X$, and so is even.

$\Sigma X$ is a sphere bundle over $X$. 
$\varphi\;\colon \Sigma X\to X$ is the projection $S(TX\times \RR)\to X$.
Let $B(TX)$ be the unit ball bundle of $TX$ and $S(TX)$ the unit sphere bundle of $TX$.
For $\Sigma X$ there is the ``upper hemisphere'' - ``lower hemisphere'' decomposition
\[ \Sigma X = B(TX)\cup_{S(TX)} B(TX)\]
where the first copy of $B(TX)$ is the upper hemisphere and the second copy of $B(TX)$ is the lower hemisphere.
The Spin$^c$ structure on $\Sigma X$ restricted to the upper hemisphere is the Spin$^c$ structure of $TX$ determined by its almost complex structure.

Let $(\sigma, \pi^*E^0, \pi^*E^1)$ be  a (compactly supported) symbol on $TX$. 
We shall assume that the support of $\sigma$ is contained in the interior of $B(TX)$.
$E_\sigma$ denotes the vector bundle on $\Sigma X$
\[ E_\sigma=\pi^*E^0\cup_\sigma \pi^*E^1\]
where $\pi^*E^0$ and $\pi^*E^1$ have been restricted to $B(TX)$ which is then identified, respectively, with the upper and lower hemispheres of $\Sigma X$. $\pi^*E^0$ and $\pi^*E^1$ are then clutched together on $S(TX)$ by the vector bundle isomorphism $\sigma$.

Since the interior of the upper hemisphere of $\Sigma X$ identifies with $TX$,
there is a push-forward map $\iota_*\;\colon K^0(TX)\to K^0(\Sigma X)$.
\begin{lemma} \label{compactify}
In $K^0(\Sigma X)$
\[ \iota_*(\sigma, \pi^*E^0, \pi^*E^1) = [E_\sigma]-[\varphi^*E^1]\]
\end{lemma}
\begin{proof}
Let $\tilde{\sigma}$ be the vector bundle map $E_\sigma \to \varphi^*E^1$
that on the lower hemisphere is the identity map of $\pi^*E^1$, and on the upper hemisphere is $\sigma$.
Hence, (by definition) $\iota_*(\sigma, \pi^*E^0, \pi^*E^1)=(\tilde{\sigma}, E_\sigma,\varphi^*E^1)$.
Since $\Sigma X$ is compact,  $(\tilde{\sigma}, E_\sigma,\varphi^*E^1)=[E_\sigma]-[\varphi^*E^1]$.

\end{proof}

The clutching map
\[ c:K^0(TX)\to K_0^{top}(X)\]
is 
\[ c(\sigma, \pi^*E^0, \pi^*E^1) = (\Sigma X, E_\sigma, \varphi)\]
{\bf Remark.} In geometric $K$-homology, the $K$-cycle $(\Sigma X, \varphi^\ast E^1, \varphi)$ bounds, and therefore is zero in $K_0^{top}(X)$.

\subsection{Analytic $K$-homology (with compact supports)}

We review analytic $K$-homology with {\em compact supports}.
For analytic $K$-homology see \citelist{\cite{HR00} \cite{Bl98}}.
\vskip 6pt

The category of topological spaces for which this theory is defined consists of compactly generated Hausdorff topological spaces for which every compact subset is second countable (i.e., metrizable).
Any $CW$ complex has this property.

For a space $X$ in our category we denote
\[ K^a_j(X) = \lim_{\stackrel{\longrightarrow}{\Delta\subset X}} KK^j(C(\Delta),\CC)\]
The direct limit is taken over compact subsets $\Delta\subseteq X$. 
In the category of CW complexes one could take the limit over all finite subcomplexes $\Delta$.
 
We briefly review the definition of the Kasparov $K$-homology groups for second countable compact Hausdorff spaces $\Delta$.
Then $C(\Delta)$ is a unital separable $C^*$-algebra.
An {\em even Fredholm module} $(T, H^0, H^1, \rho_0, \rho_1)$ consists of two Hilbert spaces $H^0, H^1$ each equipped with a $\ast$-representation $\rho_j\colon C(\Delta)\to \mathcal{L}(H^j)$,
and a bounded linear operator $T\colon H^0\to H^1$, such that for all $f\in C_0(X)$
\[ (T^*T-1)\rho_0(f),\quad (TT^*-1)\rho_1(f),\quad T\rho_0(f)-\rho_1(f)T\]
are compact operators.

An {\em odd Fredholm module}  $(T, H, \rho)$ consists of a Hilbert space $H$ equipped with a $\ast$-representations $\rho\colon C(\Delta)\to H$,
and a {\em self-adjoint} bounded linear operator $T\colon H\to H$, such that for all $f\in C_0(X)$
\[ (T^2-1)\rho(f),\quad T\rho(f)-\rho(f)T\]
are compact operators.

The group $KK^0(C(\Delta),\CC)$ consists of equivalence classes of even Fredholm modules,
while $KK^1(C(\Delta),\CC)$ consists of equivalence classes of odd Fredholm modules.
Addition is defined by direct sum of Fredholm modules.
The equivalence relation for Fredholm modules is generated by the direct sum relation and operator homotopy (for details, see \citelist{\cite{Bl98} \cite{HR00}}).

\subsection{The ``Choose an operator'' map}

We review (including excision) the connection between elliptic operators trivial at infinity and compactly supported $K$-homology.

 \vskip 6pt
 
Let $W$ be a (not necessarily compact) manifold without boundary.
 A pseudodifferential operator $P\;\colon C_c^\infty(W, E)\to C_c^\infty(W, F)$ of order zero is {\em trivial at infinity} if there exists open sets $\Omega_1, \Omega_2$ in $W$
and an isomorphism $\psi$ of vector bundles $E|\Omega_2\to F|\Omega_2$,
such that 
\begin{itemize}
\item $\Omega_1$ has compact closure $\bar{\Omega}_1$.
\item $W=\Omega_1\cup \Omega_2$.
\item If a section $s\in C_c^\infty(W,E)$ has support in $\Omega_1$ then the support of $Ps$ is in $\Omega_1$.
\item If a section $s\in C_c^\infty(W,E)$ has support in $\Omega_2$ then  $Ps=\psi s$.
\end{itemize}

 There is a well-defined ``choose an operator'' map
 \[    K^0(TW)\to KK^0_c(C_0(W),\CC)\]
 even if $W$ is not compact.
 Given a compactly supported symbol $(\sigma, \pi^*E, \pi^*F)$ on $TW$,
 i.e., an element in $K^0(TW)$,
 we construct an element in compactly supported $K$-homology $KK^0_c(C_0(W),\CC)$ as follows.
 
First, we may assume that outside a compact set  $\sigma$ is the pull-back via $\pi$ of a vector bundle isomorphism.
Granted this, we can choose a pseudodifferential elliptic operator $P$ of order zero that is trivial at infinity with symbol $\sigma$,
with properties as listed above.

If $\Delta$ is an open set with smooth boundary and compact closure that contains $\Omega_1$,
then due to the properties listed above $P$ maps $C_c^\infty(\Delta, E)$ to $C_c^\infty(\Delta, F)$.
This ``restriction of $P$ to $\Delta$'' extends to a bounded operator
\[ P_\Delta\;\colon\; L^2(\Delta,E)\to L^2(\Delta,F)\]
and therefore an element $[P_\Delta]\in KK^0(C(\Delta),\CC)$.
Note that $C(\Delta)$ acts on the Hilbert spaces $L^2(\Delta,E)$, $L^2(\Delta,F)$ by the evident multiplication operators.

With $\Delta$ and $\sigma$ fixed, a different choice of operator $Q$ will result in operators $P_\Delta$, $Q_\Delta$ that differ by a compact operator.
With $\Delta$ fixed, a homotopy of $\sigma$ can be lifted to a homotopy of operators.
Finally, with $\sigma$ and $P$ fixed, if $\Delta\subset \Delta'$, then $P_{\Delta'}\in KK^0(C(\Delta'),\CC)$ is the direct sum of the push forward of $P_\Delta \in KK^0(C(\Delta),\CC)$ with a trivial element.
In summary, the $K$-theory class of $\sigma$ determines an element in  $KK^0_c(C_0(W),\CC)$.

\vskip 6pt
We will need the following excision lemma.
\begin{lemma}\label{excision}
Let $W$ be a manifold, and let $U$ be an open subset of $W$.
Then there is commutativity in the diagram
\[ \xymatrix{ K^0(TU)\ar[r]\ar[d]  & K^0(TW) \ar[d] \\
 KK^0_c(C_0(U),\CC)\ar[r] & KK^0_c(C_0(W),\CC)              
 }
\]
\end{lemma}
\begin{proof}
We may assume that the push-forward of a symbol $\sigma$ from $TU$ to $TW$ restricts to $\sigma$ on $TU$. 
Therefore,  the operator we choose on $W$ restricts to an operator on $U$ with symbol $\sigma$,
and outside of $U$ this operator will be trivial.
\end{proof}

\subsection{The equivalence of geometric and analytic $K$-homology}

Let $\Delta$ be a finite CW complex.
As was proven in \cite{BHS10}, there is a natural isomorphism
\[ \mu\;\colon\; K^{top}_j(\Delta)\to KK^j(C(\Delta),\CC)\]
For the compactly supported theory, the isomorphism was introduced in \cite{BD91}, and is constructed as follows.
Let $D_E=E\otimes D$ be the Dirac operator $D$ for the Spin$^c$ manifold $M$, twisted by the vector bundle $E$.
Form the bounded elliptic pseudodifferential operator $T=D_E(1+D_E^*D_E)^{-1/2}$, which defines an element in 
\[ [D_E]\in KK^j(C(M), \CC)\]
Then 
\[ \mu(M,E,\varphi) = \varphi_*([D_E]).\]
where $\varphi_*$ is
\[ \varphi_*\;\colon\; K^a_j(M)\to K^a_j(X)\]
\begin{theorem}
If $X$ is a CW complex, then 
\[ \mu\;\colon\; K^{top}_j(X)\to K^a_j(X)\]
is an isomorphism.
\end{theorem}
\begin{proof}
That $\mu$ is an isomorphism for any $CW$ complex $X$ follows 
from the result for finite $CW$ complexes of \cite{BHS10}.
Both groups $K^{top}_j(X)$ and $K^a_j(X)$ are the direct limit over finite subcomplexes $\Delta\subset X$.
\end{proof}

\section{A family of index 1 operators}\label{family}
We review the properties of the family of elliptic operators of index 1 introduced in \cite{AS1}.
See also \cite{Ho71}.
We place it in the context of $KK$-theory.

\subsection{Properties of $\alpha_F$}

Let $\pi\;\colon F\to X$ be a $C^\infty$ $\RR$ vector bundle on the closed $C^\infty$ manifold $X$.
Then (once a connection has been chosen for $F$) $TF$ can be viewed as a $\CC$ vector bundle over the manifold $TX$.
The projection $TF\to TX$ is the derivative of $F\to X$.
The fibers of this vector bundle are of the form $TF_x=F_x\oplus F_x$,
which we view as the complex vector space $F_x\otimes \CC$.

Let $\alpha_F$ be a family of operators whose total space is $F$ with base space $X$,
where in each fiber $F_x$ we choose an order zero pseudodifferential elliptic operator whose principal symbol, taken as an element in $K^0(TF_x)$,  is the Bott element for $TF_x=F_x\otimes \CC$.
Since the Bott element of $TF_x$ has compact support, the family $\alpha_F$ consists of operators that are trivial at infinity.

The family $\alpha_F$ has three relevant properties.
\vskip 6pt
 {\bf Property 1.} The principal symbol of the family $\alpha_F$ is the Thom class of $TF$, viewed as a $\CC$ vector bundle on $TX$. 
\vskip 6pt

{\bf Property 2.} 
The index of $\alpha_F$, as a family, is a trivial $\CC$ line bundle on $X$.
In $KK$-theory, this can be expressed as follows.
The family $\alpha_F$ gives an element in the group
\[ KK_c(C_0(F),C(X)) :=\lim_{\stackrel{\longrightarrow}{\Delta\subseteq F}} KK(C(\Delta), C(X))\]
where $\Delta\subseteq F$ ranges over all compact subsets of $F$
and
\[ \pi_*\;\colon\; KK_c(C_0(F),C(X))\to KK(C(X),C(X))\]
maps $\alpha_F$ to the unit of the ring $KK(C(X),C(X))$.

\vskip 6pt
{\bf Property 3.}
It follows from Property 2 that if $\xi$ is any element in $KK(C(X),\CC)$, then the Kasparov product 
\[ \alpha_F\# \xi \in KK_c^0(C_0(F),\CC)\]
is an element in compactly supported $K$-homology of $F$ with
\[\pi_*(\alpha_F\# \xi) = \xi\in  KK^0(C(X),\CC)\]
See \cite{AS1} for details on $\alpha_F$.

\begin{proposition}\label{alpha}
There is commutativity in the diagram
\[ \xymatrix{   K^0(TX) \ar[d]_{\rm Op}\ar[r]^{{\rm Thom}}_{\cong}& K^0(TF)\ar[d]^{\rm Op}\\
 KK(C(X),\CC)  & KK_c^0(C_0(F),\CC) \ar[l]_{\pi_*}              
 }
\]
\end{proposition}
\begin{proof}
Let $\sigma\in K^0(TX)$ be the symbol of the elliptic operator $P$ on $X$, i.e., $P={\rm Op}(\sigma)$.
The Thom isomorphism $K^0(TX)\to K^0(TF)$ maps $\sigma$ to $\tau\# \sigma$, 
where $\tau$ is the Thom class of $TF$ as a complex vector bundle on $TX$.
Then ${\rm Op}(\tau\# \sigma)=\alpha_F\# P$
and by Property 3 above, $\pi_*(\alpha_F\# P)=P$.
\end{proof}

\subsection{H\"ormander's operators}
For clarity of exposition, we give some indication here of H\"ormander's approach \cite{Ho71} to the construction of $\alpha_F$. 

Let $V$ be a finite dimensional $\RR$ vector space with a given Euclidean inner product.
The Clifford algebra $\mathrm{Cliff}(V)$ acts on the complexification of the sum of the exterior powers  $\Lambda^*_\CC V$.
For $v\in V$, $c(v)$ denotes Clifford multiplication by $v$
\[ c(v)\;\colon\; \Lambda^{even}_\CC V \to \Lambda^{odd}_\CC V\] 
Consider the elliptic operator 
\[ \alpha=c(v)+D, \qquad D=d+d^*\]
which acts on $\CC$ valued differential forms on $V$, and maps even to odd forms.
The {\em total} symbol of $\alpha$ is $c(v+i\xi)$, which is the Bott element $\beta\in K^0(TV)$.
The operator $\alpha$ and its total symbol $\beta$ are $O(V)$ equivariant,
\[ \alpha\in KK^0_{O(V)}(C_0(V),\CC)\qquad \beta\in KK^0_{O(V)}(\CC, C_0(TV))\]
Note that here (exceptionally for this paper) $KK^0_{O(V)}(C_0(V), \CC)$ is {\em not} compactly supported.

If $F$ is an $\RR$ vector bundle on $M$, as above,
choose a Euclidean structure for $F$, thus reducing its structure group to $O(n)$.
Then due to the $O(n)$ equivariance of $\alpha$, a family of operators in the fibers of $F$ is determined,
\[ \alpha'_F \in KK^0(C_0(F), C(M))\]
To obtain the element
\[ \alpha_F \in KK^0_c(C_0(F), C(M))\]
in compactly supported $KK$-theory, realized as a family of elliptic operators trivial at infinity, as in \cite{AS1},
a perturbation must be made on the operator $\alpha$.
For details of this perturbation, see \cite{Ho71}. 

\begin{remark}
The total symbol $\beta$ of the elliptic operator $\alpha$ on $V$ is its {\em principal} symbol
when $\alpha$ is viewed as an element in the Weyl calculus.
Since $\beta$ is invertible outside the point $(0,0)$ in $TV$, 
it follows, using the Weyl calculus, that the operator $\alpha$ is Fredholm.
Moreover, by the Weyl calculus index theorem,
the index of $\alpha$ is the Bott number of its Weyl symbol $\beta$, which is 1. 
In fact, as shown in \cite{Ho71}, the null space of $\alpha$ is 1 dimensional, 
and is spanned by the Gaussian $e^{-|v|^2/2}$ (a section of the line bundle $V\times \Lambda^0_\CC V$ on $V$).  
Hence, the families index of $\alpha_F'$ is a trivial line bundle on $M$.

For details on the Weyl calculus, see \cite{Ho79}.
\end{remark}

\section{The Thom class of  $T(TM)$}\label{rotation}

Let $X$ be a closed $C^\infty$ manifold.
$T(TX)$ is a complex vector bundle over $TX$ in two different ways.
Therefore there are two Thom classes $\tau_0, \tau_1$.
One is the symbol of the Dirac operator $D_{TX}$ of $TX$,
while the other is the symbol of the family $\alpha_{TX}$.

A Riemannian metric for the closed $C^\infty$ manifold $X$ makes $TX$ into an almost complex manifold as follows.
Choose normal coordinates  $x=(x_1, \dots, x_n)$ on $X$,
and let $\xi=(\xi_1,\dots, \xi_n)$ be the induced linear coordinates on the fibers of $TX$.
Then at point $x=0$ we let
\[ J\frac{\partial}{\partial x_j} = -\frac{\partial}{\partial \xi_j}\qquad J\frac{\partial}{\partial \xi_j} = \frac{\partial}{\partial x_j}\] 
Let $\pi_0$ be the projection $\pi_0\;\colon T(TX)\to TX$,
obtained by viewing $T(TX)$ as the tangent bundle of the manifold $TX$
and using the standard projection of a vector to the point from which it emanates. 
Hence, in this way $T(TX)$ is a complex vector bundle on $TX$. 
Let $\tau_0$ be its Thom class.
Note that $\tau_0$ is the symbol of the Dirac operator $D_{TX}$ of the almost complex manifold $TX$,
\footnote{Let $W$ be a Spin$^c$ manifold, i.e., $TW\to W$ is a Spin$^c$ vector bundle, and therefore has a Thom class $\tau$.
The symbol of the Dirac operator of $W$ is this Thom class $\tau$.}
\[ \tau_0=\sigma(D_{TX})\]

With notation as in \ref{alpha} above, let $F=TX$.
Then $TF=T(TX)$ is a complex vector bundle over $TX$,
where this time the projection $\pi_1\;\colon T(TX)\to TX$ is the derivative of $TX\to X$.
This is {\em not the same} as the projection $\pi_0\;\colon T(TX)\to TX$ above.
Let $\tau_1$ be the Thom class of this complex vector bundle.
The Thom class $\tau_1$ is the symbol of the family $\alpha_{TX}$.

At first glance, it appears that there are two Thom isomorphisms
\[ K^0(TX)\stackrel{\cong}{\longrightarrow} K^0(T(TX))\;\colon\;
 \sigma\mapsto \pi_j^*\sigma \otimes \tau_j,\quad j=0,1\]
where $\tau_0, \tau_1$ are the two Thom classes introduced above.
However, these two Thom isomorphisms are equal:
\begin{lemma}\label{Thom}
Let  $[\sigma]$ be any element of $K^0(TX)$. Then in $K^0(T(TX))$
\[ [\pi_0^*\sigma \otimes \tau_0] = [\pi_1^*\sigma \otimes \tau_1]\]
\end{lemma}
\begin{proof}
Let $\pi_X\colon TX\to X$ be the usual projection,
and let $\pi_{TX}\colon T(TX)\to TX$ be the projection for the tangent bundle of $TX$.
There is a second map from $T(TX)$ to $TX$,
i.e., the derivative of $\pi_X\colon TX\to X$.
Denote $\pi_0=\pi_{TX}$ and $\pi_1=d\pi_X$.

We can view $T(TX)$ as a fiber bundle on $X$ via the projection $\pi_X\circ \pi_0\;\colon T(TX)\to X$.
We denote $W:=T(TX)$ if we view it as a fiber bundle on $X$.
The fiber $W_x$ of $W$ at a point $x\in X$ is isomorphic to a direct sum of three copies of $T_xX$,
\[ W_x\cong T_xX\oplus T_xX\oplus T_xX\]
A point $w\in W_x$ will be denoted as a triple $(w_0, w_h, w_v)$.
First, $w$ is a tangent vector to $TX$.
It emanates from a point $w_0:=\pi_0(w)\in T_xX$.
The pair $(w_h, w_v)$ denotes the vertical and horizontal components of the tangent vector $w$ to the manifold $TX$ at $w_0$.

More precisely, the derivative of $\pi_X\;\colon TX\to X$ is a map $d\pi_X\colon T(TX)\to TX$,
whose kernel consists of the ``vertical'' tangent vectors $T^{vert}TX$ to $TX$.
There is the canonical identification $T_{w_0}^{vert}(TX)\cong T_xX$,
so we have a short exact sequence
\[ 0\to T_xX\to T_{w_0}(TX)\stackrel{d\pi_{M}}{\longrightarrow} T_xX\to 0\]
We let $w_h:= d\pi_X(w)$.
Choosing a splitting of this short exact sequence we have
\[ T_{w_0}(TM)\cong T_xX\oplus T_xX\]
and we let $w_v=w-w_h$.

With this notation, the projection $\pi_{TX}\;\colon T(TX)\to TX$ is the map
\[ \pi_0(w_0, w_h, w_v):= w_0\]
while the derivative $d\pi_X\;\colon T(TX)\to TX$ is the map
\[ \pi_1(w_0, w_h, w_v) = w_h\]

We recall the standard construction of the Thom class for a complex vector bundle $\pi\;\colon F\to B$.  
On $F$, consider the vector bundles $\pi^*\Lambda^{even}F$ and $\pi^*\Lambda^{odd}F$,
the even and odd exterior powers of $F$.
The Thom class is given by the vector bundle map
\[ \tau\;\colon\; \pi^*\Lambda^{even}F\to \pi^*\Lambda^{odd}F\]
which at $v\in F_b$ is
\[ \tau(v):= \wedge_v + \iota_v\]
where $\iota_v$ is the adjoint of $\wedge_v\;\colon \alpha\mapsto v\wedge \alpha$ for the choice of some hermitian structure on $F$.

The vector bundles on $T(TX)$ used to construct the two Thom classes $\tau_0, \tau_1$
are the same, i.e., the pull-back via $\pi_X\circ \pi_0=\pi_X\circ \pi_1$ of the even and odd parts of the exterior algebra $\Lambda TX\otimes \CC$.
At a point $w=(w_0, w_h, w_v)$ in $T(TX)$,  the Thom classes $\tau_0, \tau_1$ are given by
\[ \tau_0(w)=\tau(w_v+iw_h)\qquad w_v+iw_h \in T_xX\otimes \CC\]
\[ \tau_1(w)=\tau(w_0+iw_v) \qquad w_0+iw_v \in T_xX\otimes \CC\]

Let $\rho$ be the diffeomorphism
\[ \rho\;\colon\; T(TX)\to T(TX)\;\colon\; \rho(w_0, w_h, w_v) = (w_h, -w_0, w_v)\]
$\rho$ is properly homotopic to the identity via maps $\rho_t$ given by
\[ \rho_t(w_0, w_h, w_v):= (w_0\,\cos{\pi t/2}  + w_h \sin{\pi t/ 2}, -w_0\,\sin{\pi t/2}  + w_h \cos{\pi t/ 2}, w_v)\]
i.e., the map 
\[ T(TX)\times [0,1]\to T(TX)\;\colon\; (w,t)\mapsto \rho_t(w)\]
is proper.

Note that $\tau_1(w)= i\tau_0(\rho(w))$, i.e., $\tau_1=i\rho^*\tau_0$. Also $\pi_1=\pi_0\circ \rho$. 
Therefore 
\[\pi_1^*\sigma \otimes \tau_1 = i\rho^*(\pi^*_0\sigma\otimes \tau_0)\]

\end{proof}

\begin{remark}
Consider the commutative diagram of Proposition \ref{alpha} in the special case when $F=TX$,
\[ \xymatrix{   K^0(TX) \ar[d]_{\rm Op}\ar[r]^{{\rm Thom}}_{\cong}& K^0(T(TX))\ar[d]^{\rm Op}\\
 KK(C(X),\CC)  & KK_c^0(C_0(TX),\CC) \ar[l]_{\pi_*}              
 }
\]
\end{remark}
Since the two Thom isomorphisms are equal, the commutativity of this diagram can be expressed as
\[ [P] = \pi_*[\sigma_P\otimes D_{TX}]\qquad {\rm in}\; KK^0(C(X),\CC)\] 
However, $\sigma_P\otimes D_{TX}$ has to be interpreted as an elliptic operator of order zero on $TX$ that is trivial at infinity, and whose principal symbol is homotopic to $\sigma_P\otimes \sigma(D_{TX})$.
Note that $\sigma_P\otimes \sigma(D_{TX})$ is not in any sense an operator of Dirac type.

In the next section, by passing to an appropriate compactification of $TX$,
we will replace the somewhat non-explicit operator $\sigma_P\otimes D_{TX}$ by an actual twisted Dirac operator.

\section{Proof of commutativity}\label{reduction}

In this section we give our direct proof of reduction to Dirac. 
The proof is direct, and does not use any form of the Atiyah-Singer index theorem.

Let $P$ be an elliptic (pseudo)differential operator on the closed $C^\infty$-manifold $X$,
\[ P\;\colon\; C^\infty(X,E)\to C^\infty(X,F)\]
$E$ and $F$ are complex $C^\infty$ vector bundles on $X$.
$P$ determines an element 
\[[P]=(T, H^0, H^1, \rho_0, \rho_1)\in KK^0(C(X),\CC)\]
where
\begin{itemize}
\item $T=P(1+P^*P)^{-1/2}$
\item $H^0=L^2(X,E)$ and $H^1=L^2(X,F)$
\item $\rho_j\;\colon\; C(X)\to \mathcal{L}(H^j)$ is $f\mapsto \mathcal{M}_f$ 
where $\mathcal{M}_f$ is the multiplication operator $(\mathcal{M}_f u)(p)=f(p)u(p)$.
\end{itemize}

The Dirac operator $D_{\Sigma X}$ of the Spin$^c$ manifold $\Sigma X$, twisted by the vector bundle $E_\sigma$, determines an element $[E_\sigma\otimes D_{\Sigma X}]\in KK^0(C(\Sigma X), \CC)$,
which pushes forward to 
\[ \varphi_*[E_\sigma\otimes D_{\Sigma X}]  \in KK^0(C(X),\CC)\]

\begin{theorem}
In $KK^0(C(X),\CC)$
\[ [P] = \varphi_*[E_\sigma\otimes D_{\Sigma X}] \]
i.e., there is commutativity in the diagram
\[ \xymatrix{   & K^0(TX) \ar[dr]^{\rm Op}\ar[dl]_c& \\
 K_0^{top}(X) \ar[rr]_\mu & & KK^0(C(X),\CC)              
 }
\]
\end{theorem}
\begin{proof}
Consider the diagram
\[ \xymatrix{   K^0(TX) \ar[d]_{\rm Op}\ar[r]^{{\rm Thom}}_{\cong}& K^0(T(TX))\ar[d]^{\rm Op}\ar[r]^{\iota_*}  & K^0(T(\Sigma X))\ar[d]^{\rm Op}\\
 KK(C(X),\CC)  & KK_c^0(C_0(TX),\CC) \ar[l]_{\pi_*}^{\cong}\ar [r]^{\iota_*}  & KK(C(\Sigma X),\CC)  \ar@/^2pc/[ll]_{\varphi_*}          
 }
\]
The diagram commutes by Proposition \ref{alpha} (for the square on the left), Lemma \ref{excision} (for the square on the right), and the fact that $\pi=\varphi\circ \iota$ (for the bottom row).

Starting in the upper left corner of the diagram, let $(\sigma, \pi^*E, \pi^*F)$ be the symbol of the elliptic operator $P$.
Applying the Thom isomorphism, we obtain $\tau_1\otimes \pi_1^*\sigma$ (where $\tau_1$ is the symbol of the family $\alpha_{TX}$).
By Proposition \ref{Thom} this is equal to  $\tau_0\otimes \pi_0^*\sigma$ in $K^0(T(TX))$,
where now $\tau_0$ is the symbol of the Dirac operator $D_{TX}$ of the almost complex manifold $TX$.

Since $D_{TX}$ is the restriction of the Dirac operator of $\Sigma X$ to the upper hemisphere,
Lemma \ref{compactify} implies that
\[ \iota_*(\tau_0\otimes \pi_0^*\sigma) = \sigma(D_{\Sigma X}) \otimes (E_\sigma - \varphi^* F)\]
as elements in $K^0(T(\Sigma X))$.

Commutativity of the diagram now implies that in $KK(C(X),\CC)$
\[ [P]=\varphi_*(D_{\Sigma X} \otimes E_\sigma) - \varphi_*(D_{\Sigma X}) \otimes  F\]
Finally, $\varphi_*(D_{\Sigma X})=0$ in $KK(C(X),\CC)$ because 
$\Sigma X$ is the boundary of  the ball bundle $B(TX\times \RR)$ as a Spin$^c$ manifold.

\end{proof}

\bibliographystyle{amsplain}

\bibliography{MyBibfile}

\end{document}